\documentclass[11pt]{article}
\usepackage{amsfonts, amssymb, amsmath, amsthm, xcolor, graphicx, pstricks, pstricks-add}
\begin{document}
\newcommand{\m}{\textnormal{mult}}
\newcommand{\ra}{\rightarrow}
\newcommand{\la}{\leftarrow}
\renewcommand{\baselinestretch}{1.1}

\theoremstyle{plain}
\newtheorem{thm}{Theorem}[section]
\newtheorem{cor}[thm]{Corollary}
\newtheorem{con}[thm]{Conjecture}
\newtheorem{cla}[thm]{Claim}
\newtheorem{lm}[thm]{Lemma}
\newtheorem{prop}[thm]{Proposition}
\newtheorem{example}[thm]{Example}

\theoremstyle{definition}
\newtheorem{dfn}[thm]{Definition}
\newtheorem{alg}[thm]{Algorithm}
\newtheorem{prob}[thm]{Problem}
\newtheorem{rem}[thm]{Remark}

\renewcommand{\baselinestretch}{1.1}

\title{\bf Properties of $\theta$-super positive graphs}
\author{
Cheng Yeaw Ku
\thanks{ Department of Mathematics, National University of Singapore, Singapore 117543. E-mail: matkcy@nus.edu.sg} \and K.B. Wong \thanks{
Institute of Mathematical Sciences, University of Malaya, 50603 Kuala Lumpur, Malaysia. E-mail:
kbwong@um.edu.my.} } \maketitle

\begin{abstract}\noindent
Let the matching polynomial of a graph $G$ be denoted by $\mu (G,x)$. A graph $G$ is said to be $\theta$-super positive if  $\mu(G,\theta)\neq 0$ and $\mu(G\setminus v,\theta)=0$ for all $v\in V(G)$. In particular, $G$ is $0$-super positive if and only if $G$ has a perfect matching. While much is known about $0$-super positive graphs, almost nothing is known about $\theta$-super positive graphs for $\theta \not = 0$. This motivates us to investigate the structure of $\theta$-super positive graphs in this paper. Though a $0$-super positive graph may not contain any cycle, we show that a $\theta$-super positive graph with $\theta \not = 0$ must contain a cycle. We introduce two important types of $\theta$-super positive graphs, namely $\theta$-elementary and $\theta$-base graphs. One of our main results is that any $\theta$-super positive graph $G$ can be constructed by adding certain type of edges to a disjoint union of $\theta$-base graphs; moreover, these $\theta$-base graphs are uniquely determined by $G$. We also give a characterization of $\theta$-elementary graphs: a graph $G$ is $\theta$-elementary if and only if the set of all its $\theta$-barrier sets form a partition of $V(G)$. Here, $\theta$-elementary graphs and $\theta$-barrier sets can be regarded as $\theta$-analogue of elementary graphs and Tutte sets in classical matching theory.
\end{abstract}

\bigskip\noindent
{\sc keywords:} matching polynomial, Gallai-Edmonds decomposition, elementary graph, barrier sets, extreme sets

\section{Introduction}

We begin by introducing matching polynomials with an interest in the multiplicities of their roots. This will lead us to a recent extension of the celebrated Gallai-Edmonds Strcuture Theorem by Chen and Ku \cite{KC} which will be useful later in our study of $\theta$-super positive graphs.

All the graphs in this paper are simple and finite. The vertex set and edge set of a graph $G$ will be denoted by $V(G)$ and $E(G)$, respectively.

\begin {dfn}\label {I:D1}  An $r$-\emph {matching} in a graph $G$ is a set of $r$ edges, no two of which have a vertex in common. The number of $r$-matchings in $  G$ will be denoted by $p( G,r)$. We set $p(G,0)=1$ and define the \emph {matching polynomial} of $G$ by
\begin {equation}
\mu ( G,x)=\sum_{r=0}^{\lfloor n/2\rfloor} (-1)^rp(G,r)x^{n-2r}.\notag
\end {equation}
We denote the multiplicity of $\theta$ as a root of $\mu(G,x)$ by $\textnormal {mult} (\theta,G)$. Let $u\in V(G)$, the graph obtained from $G$ by deleting the vertex $u$ and all edges that contain $u$ is  denoted by $G\setminus u$. Inductively if $u_1,\dots, u_k\in V(G)$, $G\setminus u_1\cdots u_k=(G\setminus u_1\cdots u_{k-1})\setminus u_k$. Note that the order in which the vertices are being deleted is not important, that is, if $i_1,\dots, i_k$ is a permutation  of $1,\dots, k$, we have  $G\setminus u_1\cdots u_k= G\setminus u_{1_1}\cdots u_{i_k}$. Furthermore, if  $X=\{u_1,\dots, u_k\}$, we set $G\setminus X=G\setminus u_1\cdots u_k$. If $H$ is a subgraph of $G$, by an abuse of notation, we have $G \setminus H = G \setminus V(H)$. For example, if $p=v_1v_2\dots v_n$ is a path in $G$ then $G\setminus p=G\setminus v_1v_2\cdots v_n$. If $e$ is an edge of $G$, let $G-e$ denote the graph obtained from $G$ by deleting the edge $e$ from $G$. Inductively, if $e_{1}, \ldots, e_{k} \in E(G)$, $G-e_{1}\cdots e_{k} = (G-e_{1} \cdots e_{k-1} )-e_{k}$.

A graph $G$ is said to have a \emph{perfect matching} if it has a $n/2$-matching ($n$ must be even). This is equivalent to  $\textnormal{mult} (0,G)=0$, that is, $0$ is not a root of $\mu(G,x)$. Recall that in the literature $\m(0,G)$ is also known as the {\em deficiency} of $G$ which is the number of vertices of $G$ missed by some maximum matching.
\end {dfn}

The following are some basic properties of $\mu (G,x)$.

\begin {thm}\label{matching_property} \textnormal { \cite[Theorem 1.1 on p. 2] {G0}}
\begin {itemize}
\item [(a)] $\mu (  G\cup  H,x)=\mu (  G,x)\mu (  H,x)$ where $  G$ and $  H$ are disjoint graphs,
\item [(b)] $\mu (  G,x)=\mu (  G - e,x)-\mu (  G\setminus uv, x)$ if $e=(u,v)$ is an edge of $  G$,
\item [(c)] $\mu (  G,x)=x\mu (  G\setminus u,x)-\sum_{i\sim u} \mu (  G\setminus ui,x)$ where $i\sim u$ means $i$ is adjacent to $u$,
\item [(d)] $\displaystyle \frac {d}{dx} \mu (  G,x)=\sum_{i\in V(  G)} \mu (  G\setminus i,x)$ where $V(  G)$ is the vertex set of $  G$.
\end {itemize}
\end {thm}

It is well known that all roots of $\mu(G,x)$ are real. Throughout, let $\theta$ be a real number. The multiplicity of a matching polynomial root satisfies the the following interlacing property:

\begin {lm}\label{basic_inequality}\textnormal{\cite[Corollary 1.3 on p. 97]{G0} (Interlacing)} Let $  G$ be a graph  and $u\in V(  G)$. Let $\theta$ be a real number. Then
\begin {equation}
\textnormal {mult} (\theta,   G)-1\leq \textnormal {mult} (\theta,   G\setminus u)\leq \textnormal {mult} (\theta,  G)+1.\notag
\end {equation}
\end {lm}

Lemma \ref{basic_inequality} suggests that given any real number $\theta$, we can classify the vertices of a graph according to an increase of $1$ or a decrease of $1$ or no change in the multiplicity of $\theta$ upon deletion of a vertex.

\begin {dfn}\label {P:D3}\textnormal { \cite [Section 3]{G}} For any $u\in V(  G)$,
\begin {itemize}
\item [(a)] $u$ is $\theta$-\emph {essential} if $\textnormal {mult} (\theta,   G\setminus u)=\textnormal {mult} (\theta,   G)-1$,
\item [(b)] $u$ is $\theta$-\emph {neutral} if $\textnormal {mult} (\theta,   G\setminus u)=\textnormal {mult} (\theta,   G)$,
\item [(c)] $u$ is $\theta$-\emph {positive} if $\textnormal {mult} (\theta,   G\setminus u)=\textnormal {mult} (\theta,   G)+1$.
\end {itemize}
Furthermore, if $u$ is not $\theta$-essential but it is adjacent to some $\theta$-essential vertex, we say $u$ is $\theta$-special.
\end {dfn}

It turns out that $\theta$-special vertices play an important role in the Gallai-Edmonds Decomposition of a graph (see \cite {KC}). Godsil \cite[Corollary 4.3]{G} proved that a $\theta$-special vertex must be $\theta$-positive. Note that if $\textnormal {mult} (\theta,   G)=0$ then for any $u\in V(G)$, $u$ is either $\theta$-neutral or $\theta$-positive and no vertices in $G$ can be $\theta$-special. Now $V(G)$ can be partitioned into the following sets:
\begin {equation}
V(G)=D_{\theta}(G)\cup A_{\theta}(G)\cup P_{\theta}(G)\cup N_{\theta}(G),\notag
\end {equation}
where
\begin {itemize}
\item [] $D_{\theta}(G)$ is the set of all $\theta$-essential vertices in $G$,
\item [] $A_{\theta}(G)$ is the set of all $\theta$-special vertices in $G$,
\item [] $N_{\theta}(G)$ is the set of all $\theta$-neutral vertices in $G$,
\item [] $P_{\theta}(G)=Q_{\theta}(G)\setminus A_{\theta}(G)$, where  $Q_{\theta}(G)$ is the set of all $\theta$-positive vertices in $G$.
\end {itemize}
Note that there are no $0$-neutral vertices. So $N_0(G)=\varnothing$ and $V(G)=D_{0}(G)\cup A_{0}(G)\cup P_{0}(G)$.

\begin {dfn}\label {P:D4}\textnormal {\cite [Section 3]{G}} A graph $G$ is said to be $\theta$-critical if all vertices in $G$ are $\theta$-essential and $\textnormal {mult} (\theta, G)=1$.
\end {dfn}

The celebrated Gallai-Edmonds Structure Theorem describes the stability of a certain canonical decomposition of $V(G)$ with respect to the zero root of $\mu (G,x)$. In \cite {KC},  Chen and Ku extended the Gallai-Edmonds Structure Theorem to any root $\theta \not = 0$, which consists of the following two theorems:

\begin {thm}\label{Gallai_Edmond_vertex}\textnormal { \cite [Theorem 1.5]{KC} ($\theta$-Stability Lemma)} Let $G$ be a graph with $\theta$ a root of $\mu (G,x)$. If $u\in A_{\theta} (G)$ then
\begin {itemize}
\item [(i)] $D_{\theta}(G\setminus u)=D_{\theta}(G)$,
\item [(ii)] $P_{\theta}(G\setminus u)=P_{\theta}(G)$,
\item [(iii)] $N_{\theta}(G\setminus u)=N_{\theta}(G)$,
\item [(iv)] $A_{\theta}(G\setminus u)=A_{\theta}(G)\setminus \{u\}$.
\end {itemize}
\end {thm}

\begin {thm}\label{Gallai_lemma}\textnormal {\cite[Theorem 1.7]{KC} ($\theta$-Gallai's Lemma)} If $G$ is connected and every vertex of $G$ is $\theta$-essential then $\textnormal {mult} (\theta,   G)=1$.
\end {thm}

Theorem \ref{Gallai_Edmond_vertex} asserts that the decomposition of $V(G)$ into $D_{\theta}(G)$, $P_{\theta}(G)$, $N_{\theta}(G)$ and $A_{\theta}(G)$ is {\em stable} upon deleting a $\theta$-special vertex of $G$. We may delete every such vertex one by one until there are no $\theta$-special vertices left. Together with Theorem \ref{Gallai_lemma}, it is not hard to deduce the following whose proof is omitted.

\begin {cor}\label{Gallai_Edmond_decom}\
\begin {itemize}
\item [(i)]  $A_{\theta}(G\setminus A_{\theta}(G))=\varnothing$, $D_{\theta}(G\setminus A_{\theta}(G))=D_{\theta}(G)$, $P_{\theta}(G\setminus A_{\theta}(G))=P_{\theta}(G)$, and $N_{\theta}(G\setminus A_{\theta}(G))=N_{\theta}(G)$.
\item [(ii)] $G\setminus A_{\theta}(G)$ has exactly $\vert A_{\theta}(G)\vert+\textnormal {mult} (\theta, G)$ $\theta$-critical components.
\item [(iii)] If $H$ is a component of $G\setminus A_{\theta}(G)$ then either $H$ is $\theta$-critical or $\textnormal {mult} (\theta, H)=0$.
\item [(iv)] The subgraph induced by $D_{\theta}(G)$ consists of all the  $\theta$-critical components in $G\setminus A_{\theta}(G)$.
\end {itemize}
\end {cor}

This paper is devoted to the study of $\theta$-super positive graphs. A graph is $\theta$-super positive if $\theta$ is not a root of $\mu(G,x)$ but is a root of $\mu(G \setminus v, x)$ for every $v \in V(G)$. It is worth noting that $G$ is $0$-super positive if and only if $G$ has a perfect matching. While much is known about graphs with a perfect matching, almost nothing is known about $\theta$-super positive graphs for $\theta \not = 0$. This gives us a motivation to investigate the structure of these graphs.

The outline of this paper is as follows:

In Section 2, we show how to construct $\theta$-super positive graphs from smaller $\theta$-super positive graphs (see Theorem \ref{construction_positive}). We prove that a tree is $\theta$-super positive if and only if $\theta=0$ and it has a perfect matching (see Theorem \ref{tree_super_positive}). Consequently, a $\theta$-super positive graph must contain a cycle when $\theta\neq 0$. For a connected vertex transitive graph $G$, we prove that it is $\theta$-super positive for any root $\theta$ of $\mu (G\setminus v, x)$ where $v\in V(G)$ (see Theorem \ref{vertex_transitive_positive}). Finally we prove that if $G$ is $\theta$-super positive, then $N_{\theta}(G\setminus v)=\varnothing$ for all $v\in V(G)$ (see Theorem \ref{neutral_empty_positive}).

In Section 3, we introduce $\theta$-elementary graphs. These are $\theta$-super positive graphs with $P_{\theta}(G \setminus v) = \emptyset$ for all $v \in V(G)$. We prove a characterization of $\theta$-elementary graphs: a graph $G$ is $\theta$-elementary if and only if the set of all $\theta$-barrier sets form a partition of $V(G)$ (see Theorem \ref{elementary_characterization}).

In Section 4, we apply our results in Section 3 to prove that an $n$-cycle $C_n$ is 1-elementary if and only if $n=3k$ for some $k\in\mathbb N$ (see Theorem \ref{1_elementary_cycle}). Furthermore, we prove that $C_{3k}$ has exactly 3 1-barrier sets (see Corollary \ref{1_elementary_cycle_partition}).

In Section 5, we introduce $\theta$-base graphs which can be regarded as building blocks of $\theta$-super positive graphs. We prove a characterization of $\theta$-super positive graphs, namely a $\theta$-super positive graph can be constructed from a disjoint union of $\theta$-base graphs by adding certain type of edges; moreover, these $\theta$-base graphs are uniquely determined by $G$ (see Theorem \ref{decom_super_2} and Corollary \ref{construction_positive1}).

\section{$\theta$-super positive graphs}

\begin{dfn}\label{super_positive_dfn} A graph $G$ is $\theta$-{\em super positive} if $\theta$ is not a root of $\mu(G,x)$ and every vertex of $G$ is $\theta$-positive.
\end{dfn}

By Lemma \ref{basic_inequality}, this is equivalent to  $\textnormal {mult} (\theta, G)=0$ and $\textnormal {mult} (\theta, G\setminus v)=1$ for all $v\in V(G)$. There are a lot of $\theta$-super positive graphs. For instance the three cycle, $C_3$ and the six cycle, $C_6$ are 1-super positive. In the next theorem, we  will show how to construct $\theta$-super positive graphs from smaller $\theta$-super positive graphs.

\begin{thm}\label{construction_positive} Let $G_1$ and $G_2$ be two $\theta$-super positive graphs and $v_i\in V(G_i)$ for $i=1,2$. Let $G$ be the graph obtained by adding the edge $(v_1,v_2)$ to the union of $G_1$ and $G_2$. Then $G$ is $\theta$-super positive.
\end{thm}

\begin{proof} Let $e=(v_1,v_2)$. First we prove that $\mu (G,\theta)\neq 0$. By part (b) of Theorem \ref{matching_property}, we have $\mu (G,x)=\mu (G - e,x)-\mu ( G\setminus v_1v_2, x)$. It then follows from part (a) of Theorem \ref{matching_property} that $\mu (G,x)=\mu (G_1,x)\mu (G_2,x)-\mu (G_1\setminus v_1, x)\mu (G_2\setminus v_2, x)$. Since $G_1$ and $G_2$ are $\theta$-super positive,  $\mu (G,\theta)=\mu (G_1,\theta)\mu (G_2,\theta)\neq 0$.

It is left to prove that  $\mu (G\setminus v,\theta)=0$ for all $v\in V(G)$. Let $v\in V(G_1)$. Suppose $v=v_1$. Then by part (a) of Theorem \ref{matching_property}, $\mu (G\setminus v,x)=\mu (G_1\setminus v_1,x)\mu (G_2,x)$, and thus $\mu (G\setminus v,\theta)=0$. Suppose $v\neq v_1$. By part (b) of Theorem \ref{matching_property}, $\mu (G\setminus v,x)=\mu ((G\setminus v)- e,x)-\mu ((G\setminus v)\setminus v_1v_2, x)$. Note that $(G\setminus v)-e=(G_1\setminus v)\cup G_2$ and $(G\setminus v)\setminus v_1v_2=(G_1\setminus vv_1)\cup (G_2\setminus v_2)$. Hence $\mu (G\setminus v,\theta)=\mu (G_1\setminus v,\theta)\mu (G_2,\theta)-\mu (G_1\setminus vv_1, \theta)\mu (G_2\setminus v_2, \theta)=0$ (part (a) of Theorem \ref{matching_property}).

The case $v\in V(G_2)$ is proved similarly.
\end{proof}

The graph $G$ in Figure 1 is constructed by using Theorem \ref{construction_positive}, with $G_1=C_6$ and $G_2=C_3$. Therefore it is 1-super positive graph.

\begin{center}
\begin{pspicture}(0,0)(7,4)
\cnodeput(1, 2){1}{}
\cnodeput(2, 1){2}{}
\cnodeput(2, 3){3}{}
\cnodeput(3, 1){4}{}
\cnodeput(3, 3){5}{}
\cnodeput(4, 2){6}{}
\cnodeput(5, 2){7}{}
\cnodeput(6, 1.5){8}{}
\cnodeput(6, 2.5){9}{}
\ncline{1}{2}
\ncline{1}{3}
\ncline{2}{4}
\ncline{4}{6}
\ncline{6}{5}
\ncline{5}{3}
\ncline{6}{7}
\ncline{7}{8}
\ncline{7}{9}
\ncline{8}{9}
\rput(3.5,0){Figure 1.}
\rput(0,2){$G=$}
\end{pspicture}
\end{center}

It is clear that a $0$-super positive may or may not contain any cycle. However, we will show later that if $G$ is $\theta$-super positive and $\theta\neq 0$, then it must contain a cycle (see Corollary \ref{cycle_super_positive}). Note that any tree $T$ with at least three vertices can be represented in the following form (see Figure 2), where $u$ is a vertex with $n+1$ neighbors $v_{1}$, $\ldots$, $v_{n+1}$ such that all of them except possibly $v_{1}$ have degree $1$ and $T_1$ is a subtree of $T$ that contains $v_1$. Such a representation of $T$ is denoted by $(T_{1}, u; v_{1}, \ldots, v_{n+1})$.

\begin{center}
\begin{pspicture}(0,0)(5,5)
\cnodeput(0.5, 2){1}{}
\cnodeput(1.5, 2){2}{}
\cnodeput(2.5, 2){3}{}
\cnodeput(4.5, 2.6){4}{}
\cnodeput(2, 3){5}{}
\cnodeput(2, 4){6}{}
\ncline{1}{5}
\ncline{2}{5}
\ncline{3}{5}
\ncline{4}{5}
\ncline{5}{6}
\psline[linestyle=dashed,linearc=1]{-}(2.8,2)(3.5,2)(4.2,2.4)
\psline[linestyle=dashed]{-}(2.2,2.8)(2.8,2.1)
\psline[linestyle=dashed]{-}(2.3,2.9)(4.1,2.5)
\psline[linestyle=dashed]{-}(2.25,2.85)(3.4,2.2)
\psellipse[](0.5,1.5)(0.5,1)
\rput(0.5,1){$T_1$}
\rput(0.4,1.7){$v_1$}
\rput(1.4,1.7){$v_2$}
\rput(2.4,1.7){$v_3$}
\rput(4.8,2.3){$v_n$}
\rput(1.8,3.2){$u$}
\rput(2.7,4){$v_{n+1}$}
\rput(-1,3){$T=$}
\rput(2.5,0){Figure 2.}
\end{pspicture}
\end{center}

\begin{lm}\label{tree_root} Let $T$ be a tree with at least three vertices. Suppose $T$ has a representation $(T_{1}, u; v_{1}, \ldots, v_{n+1})$. Then $\theta$ is a root of $\mu (T,x)$ if and only if
\begin {equation}
(n-\theta^2)\theta^{n-1} \mu (T_1,\theta)+\theta^n\mu (T_1\setminus v_1,\theta)=0.\notag
\end {equation}
\end{lm}

\begin{proof} By part (c) of Theorem \ref{matching_property}, $\mu (T,\theta)=\theta\mu (T\setminus u,\theta)-\sum_{i=1}^{n+1}\mu (T\setminus uv_i,\theta)$ (see Figure 2), which implies (using part (a) of Theorem \ref{matching_property}),
\begin {equation}
\mu (T,\theta)=(\theta^2-n)\theta^{n-1}\mu (T_1,\theta)-\theta^{n}\mu (T_1\setminus v,\theta).\notag
\end {equation}
Hence the lemma holds
\end{proof}

\begin{thm}\label{tree_super_positive} Let $T$ be a tree. Then $T$ is $\theta$-super positive  if and only if $\theta=0$ and it has a perfect matching.
\end{thm}

\begin{proof} Suppose $T$ is $\theta$-super positive and $\theta\neq 0$. Then $T$ must have at least three vertices. By Lemma \ref{tree_root},
\begin {equation}
(n-\theta^2)\theta^{n-1} \mu (T_1,\theta)+\theta^n\mu (T_1\setminus v_1,\theta)\neq 0.\notag
\end {equation}
By part (a) of Theorem \ref{matching_property}, $0=\mu (T\setminus u,\theta)=\theta^{n}\mu (T_1,\theta)$ (see Figure 2). Therefore $\mu (T_1,\theta)=0$ and $ \mu (T_1\setminus v_1,\theta)\neq 0$.  Now $\mu (T\setminus v_{n+1},\theta)=0$. By part (c) of  of Theorem \ref{matching_property}, $\mu (T\setminus v_{n+1},\theta)=\theta\mu (T\setminus uv_{n+1},\theta)-\sum_{i=1}^n \mu (T\setminus uv_iv_{n+1},\theta)=\theta^{n}\mu (T_1,\theta)-(n-1)\theta^{n-2}\mu (T_1,\theta)-\theta^{n-1}\mu (T_1\setminus v_1,\theta)$. This implies that $\mu (T_1\setminus v_1,\theta)=0$, a contradiction. Hence $\theta=0$. Since 0 is not a root of $\mu (T,x)$, $T$ must have a perfect matching.

The converse is obvious.
\end{proof}

A consequence of Theorem \ref{tree_super_positive} is the following corollary.
\begin{cor}\label{cycle_super_positive} If $G$ is $\theta$-super positive for some $\theta\neq 0$, then $G$ must contain a cycle.
\end{cor}

We shall need the following lemmas.

\begin{lm}\label{heilmann_lieb}\textnormal {\cite[Theorem 6.3]{HL}(Heilmann-Lieb Identity)} Let $u,v\in V(G)$. Then
\begin{equation}
\mu(G\setminus u, x)\mu(G\setminus v,x)-\mu(G,x)\mu(G\setminus uv)=\sum_{p\in \mathcal P(u,v)} \mu(G\setminus p,x)^2,\notag
\end{equation}
where $\mathcal P(u,v)$ is the set of all the paths from $u$ to $v$ in $G$.
\end{lm}

\begin{lm}\label{existence_essential}\textnormal {\cite[Lemma 3.1]{G}} Suppose $\m(\theta, G)>0$. Then $G$ contains at least one $\theta$-essential vertex.
\end{lm}

\begin{thm}\label{vertex_transitive_positive} Let $G$ be connected, vertex transitive and $z\in V(G)$. If $\theta$ is a root of $\mu (G\setminus z,x)$ then $G$ is $\theta$-super positive.
\end{thm}

\begin{proof} Since $G\setminus z$ is isomorphic to $G\setminus y$ for all $y\in V(G)$, $\mu (G\setminus z,x)=\mu (G\setminus y,x)$ for all $y\in V(G)$. So  $\textnormal {mult} (\theta, G\setminus z)=\textnormal {mult} (\theta, G\setminus y)$. This implies that $\theta$ is a root of $\mu (G\setminus y,x)$ for all $y$.

Now it remains to show that $\mu (G,\theta)\neq 0$. Suppose the contrary. Then by Lemma \ref{existence_essential}, $G$ has at least one $\theta$-essential vertex. Since $G$ is vertex transitive, all vertices in $G$ are $\theta$-essential. By  Theorem \ref{Gallai_lemma}, $\textnormal {mult} (\theta, G)=1$. But then  $\textnormal {mult} (\theta, G\setminus z)=0$, a contradiction. Hence $\mu (G,\theta)\neq 0$ and $G$ is $\theta$-super positive.
\end{proof}

However, a $\theta$-super positive graph is not necessarily vertex transitive (see Figure 1). Furthermore a $\theta$-super positive graph is not necessary connected, for the union of two $C_3$ is $1$-super positive.

\begin{thm}\label{neutral_empty_positive} Let $G$ be $\theta$-super positive. Then $N_{\theta}(G\setminus v)=\varnothing$ for all $v\in V(G)$.
\end{thm}

\begin{proof} Suppose $N_{\theta}(G\setminus v)\neq\varnothing$ for some $v\in V(G)$. Let $u\in N_{\theta}(G\setminus v)$. By Lemma \ref{heilmann_lieb},
\begin{equation}
\mu(G\setminus u, x)\mu(G\setminus v,x)-\mu(G,x)\mu(G\setminus uv)=\sum_{p\in \mathcal P(u,v)} \mu(G\setminus p,x)^2.\notag
\end{equation}
Note that  the multiplicity of $\theta$ as a root of $\mu(G\setminus u,x)\mu(G\setminus v,x)$ is $2$, while the multiplicity of $\theta$ as a root of $\mu(G,x)\mu(G\setminus vu,x)$ is $1$ since $u$ is $\theta$-neutral in $G \setminus v$. Therefore the multiplicity of $\theta$ as a root of the polynomial on the left-hand side of the equation is at least $1$. But the multiplicity of $\theta$ as a root of the polynomial on the right-hand side of the equation is even and so, in comparison with the left-hand side, it must be at least $2$. This forces the multiplicity of $\theta$ as a root of $\mu(G,x)\mu(G\setminus vu,x)$ to be at least $2$, a contradiction. Hence $N_{\theta}(G\setminus v)=\varnothing$ for all $v\in V(G)$.
\end{proof}

Now we know that for a $\theta$-super positive graph $G$, $N_{\theta}(G\setminus v)=\varnothing$ for all $v\in V(G)$. So it is quite natural to ask whether $P_{\theta}(G\setminus v)=\varnothing$ for all $v\in V(G)$. Well, this is not true in general (see Figure 1). This motivates us to study  the $\theta$-super positive graph $G$, for which $P_{\theta}(G\setminus v)=\varnothing$ for all $v\in V(G)$. We proceed to do this in the next section.

\section{$\theta$-elementary graphs}

\begin{dfn}\label{ext_elementary_def} A graph  $G$ is said to be $\theta$-\emph{elementary} if it is $\theta$-super positive and $P_{\theta}(G\setminus v)=\varnothing$ for all $v\in V(G)$.
\end{dfn}

The graph $G$ in Figure 3 is 1-elementary. Not every $\theta$-positive graph is $\theta$-elementary. For instance, the graph in Figure 1 is not 1-elementary.

\begin{center}
\begin{pspicture}(0,0)(6,5)
\cnodeput(2, 4){1}{}
\cnodeput(5, 4){2}{}
\cnodeput(3, 3){3}{}
\cnodeput(4, 3){4}{}
\cnodeput(2, 2){5}{}
\cnodeput(5, 2){6}{}
\ncline{1}{3}
\ncline{1}{4}
\ncline{1}{5}
\ncline{2}{3}
\ncline{2}{4}
\ncline{2}{6}
\ncline{3}{4}
\ncline{3}{5}
\ncline{4}{6}
\rput(3.5,1){Figure 3.}
\rput(1.2,3){$G=$}
\rput(2.3,4.2){$u_1$}
\rput(4.7,4.2){$u_2$}
\rput(3.1,2.7){$u_3$}
\rput(4,2.7){$u_4$}
\rput(2.4,1.8){$u_5$}
\rput(4.7,1.8){$u_6$}
\end{pspicture}
\end{center}

\begin{thm}\label{elementary_characterization1} A graph $G$ is $\theta$-elementary if and only if $\textnormal {mult} (\theta, G)=0$ and $P_{\theta}(G\setminus v)\cup N_{\theta}(G\setminus v)=\varnothing$ for all $v\in V(G)$.
\end{thm}

\begin{proof} Suppose $\textnormal {mult} (\theta, G)=0$ and $P_{\theta}(G\setminus v)\cup N_{\theta}(G\setminus v)=\varnothing$ for all $v\in V(G)$. Then for each $v\in V(G)$, $\textnormal {mult} (\theta, G\setminus v)=1$, for otherwise $G \setminus v$ would only consist of $\theta$-neutral and $\theta$-positive vertices whence $P_{\theta}(G\setminus v)\cup N_{\theta}(G\setminus v)\neq \varnothing$. Therefore $G$ is $\theta$-super positive and it is $\theta$-elementary.

The other implication follows from Theorem \ref{neutral_empty_positive}.
\end{proof}

It turns out that the notion of a  $0$-elementary graph coincide with the classical notion of an elementary graph. Properties of elementary graphs can be found in Section 5.1 on p. 145 of \cite{Lo}.

The number of $\theta$-critical components in $G$ is denoted by $c_{\theta} (G)$.

\begin {dfn}\label{barrier_def}  A $\theta$-\emph {barrier set} is defined to be a set $X\subseteq V(G)$ for which $ \textnormal {mult} (\theta,G)=c_{\theta} (G\setminus X)-\vert X\vert$.

A $\theta$-\emph {extreme set} is defined to be a set $X\subseteq V(G)$ for which  $\textnormal {mult} (\theta,G\setminus X)=\textnormal {mult} (\theta,G)+\vert X\vert$.
\end {dfn}

$\theta$-barrier sets and $\theta$-extreme sets can be regarded as $\theta$-analogue of Tutte sets and extreme sets in classical matching theory. Properties of $\theta$-barrier sets and $\theta$-extreme sets have been studied by Ku and Wong \cite {KW}. In particular, the following results are needed.

\begin{lm}\label{subset_extreme}\textnormal{\cite[Lemma 2.5]{KW}} A subset of a $\theta$-extreme set is a $\theta$-extreme set.
\end{lm}

\begin {lm}\label{barrier_exclusion}\textnormal{\cite[Lemma 2.6]{KW}} If $X$ is a $\theta$-barrier set and $Y\subseteq X$ then $X\setminus Y$ is a $\theta$-barrier  set in $G\setminus Y$.
\end {lm}

\begin {lm}\label{barrier_contains_extreme}\textnormal{\cite[Lemma 2.7]{KW}} Every $\theta$-extreme set of $G$ lies in a $\theta$-barrier set.
\end {lm}

\begin {lm}\label{barrier_is_extreme}\textnormal{\cite[Lemma 2.8]{KW}} Let $X$ be a $\theta$-barrier set. Then $X$ is a $\theta$-extreme set.
\end {lm}

\begin {lm}\label{barrier_special_0}\textnormal{\cite[Lemma 3.1]{KW}} If $X$ is a $\theta$-barrier set then $X\subseteq A_{\theta}(G)\cup P_{\theta}(G)$.
\end {lm}

\begin {lm}\label{barrier_characterization}\textnormal{\cite[Theorem 3.5]{KW}} Let $X$ be a $\theta$-barrier set in $G$. Then $A_{\theta}(G)\subseteq X$.
\end {lm}

\begin{lm}\label{barrier_special} Let $G$ be a graph. If $X$ is a $\theta$-barrier set in $G$, $x\in X$ and $P_{\theta}(G\setminus x)=\varnothing$, then $A_{\theta}(G\setminus x)=X\setminus x$.
\end{lm}

\begin{proof} By Lemma \ref{barrier_exclusion}, $X\setminus x$ is a $\theta$-barrier set in $G\setminus x$. By Lemma \ref{barrier_special_0}, $X\setminus x\subseteq A_{\theta}(G\setminus x)\cup P_{\theta}(G\setminus x)$. Therefore  $X\setminus x\subseteq A_{\theta}(G\setminus x)$. It then follows from Lemma \ref{barrier_characterization} that $A_{\theta}(G\setminus x)=X\setminus x$.
\end{proof}

\begin {dfn}\label{barrier_partition_def} We define $\mathfrak P(\theta,G)$ to be the set of all the $\theta$-barrier sets in $G$.
\end {dfn}

Note that in Figure 3, $\mathfrak P(1,G)=\{\{u_1\},\{u_2\},\{u_3,u_4\},\{u_5\},\{u_6\} \}$. Now Lemma \ref{positive_lemma1} follows from part (c) of Theorem \ref{matching_property}.
\begin{lm}\label{positive_lemma1} Suppose $G$ is $\theta$-super positive. Then for each $v\in V(G)$ there is a $u\in V(G)$ with $(u,v)\in E(G)$ and $\textnormal{mult} (\theta, G\setminus uv)=0$.
\end{lm}

\begin{thm}\label{elementary_characterization} A graph $G$ is $\theta$-elementary if and only if $\mathfrak P(\theta,G)$ is a partition of $V(G)$.
\end{thm}

\begin{proof} Let $\mathfrak P(\theta,G)=\{S_1,\dots, S_k\}$.

\noindent
($\Rightarrow $) Suppose $G$ is $\theta$-elementary. Then for each $v\in V(G)$, $\{v\}$ is a $\theta$-extreme set. By Lemma \ref{barrier_contains_extreme}, it is contained in some $\theta$-barrier set. Therefore $V(G)=S_1\cup\cdots\cup S_k$. It remains to prove that $S_i\cap S_j=\varnothing$ for $i\neq j$. Suppose the contrary. Let $x\in S_i\cap S_j$. By Lemma \ref{barrier_special}, $S_i\setminus \{x\}=A_{\theta}(G\setminus x)=S_j\setminus \{x\}$ and so $S_i=S_j$, a contradiction. Hence $S_i\cap S_j=\varnothing$ for $i\neq j$ and $\mathfrak P(\theta,G)$ is a partition of $V(G)$.

\noindent
($\Leftarrow $) Suppose $\mathfrak P(\theta,G)$ is a partition of $V(G)$. Let $v\in V(G)$. Then $v\in S_i$ for some $\theta$-barrier set $S_i$. By Lemma \ref{barrier_special_0}, $v\in A_{\theta}(G)\cup P_{\theta}(G)$. Therefore $V(G)\subseteq A_{\theta}(G)\cup P_{\theta}(G)$. This implies that $\textnormal {mult} (\theta, G)=0$, for otherwise $D_{\theta}(G)\neq \varnothing$ by Lemma \ref{existence_essential}. Hence $A_{\theta}(G)=\varnothing$ and  $V(G)= P_{\theta}(G)$, i.e., $G$ is $\theta$-super positive. It remains to show that  $P_{\theta}(G\setminus v)=\varnothing$ for all $v\in V(G)$. Suppose the contrary. Then $P_{\theta}(G\setminus v_0)\neq\varnothing$ for some $v_0\in V(G)$. We may assume $v_0\in S_1$. By Corollary \ref{Gallai_Edmond_decom}, $(G\setminus v_0)\setminus A_{\theta}(G\setminus v_0)$ has a component $H$ for which $\textnormal {mult} (\theta,H)=0$. By Theorem \ref{neutral_empty_positive}, $N_{\theta}(G\setminus v_0)=\varnothing$. So we conclude that $H$ is $\theta$-super positive. Let $w\in H$. By Lemma \ref{positive_lemma1}, there is a $z\in V(H)$  with $(w,z)\in E(H)$ and $\textnormal{mult} (\theta, H\setminus wz)=0$. By part (a) of Theorem \ref{matching_property}, and, (ii) and (iii) of Corollary \ref{Gallai_Edmond_decom}, $\textnormal{mult} (\theta, ((G\setminus v_0)\setminus A_{\theta}(G\setminus v_0))\setminus wz)=1+\vert A_{\theta}(G\setminus v_0)\vert$.

On the other hand, by Lemma \ref{barrier_exclusion}, $S_1\setminus \{v_0\}$ is a $\theta$-barrier set in $G\setminus v_0$. So by Lemma \ref{barrier_characterization},  $A_{\theta}(G\setminus v_0)\subseteq S_1\setminus \{v_0\}$. By Lemma \ref{barrier_exclusion} again, $S_1\setminus (\{v_0\}\cup A_{\theta}(G\setminus v_0))$ is a $\theta$-barrier set in $(G\setminus v_0)\setminus A_{\theta}(G\setminus v_0)$. Note that $w$ is $\theta$-positive in $G\setminus v_0$ (by Corollary \ref{Gallai_Edmond_decom}). Therefore $\{w,v_0\}$ is an $\theta$-extreme set. By Lemma \ref{barrier_contains_extreme}, $\{w,v_0\}$ is contained in some $\theta$-barrier set. Since $\mathfrak P(\theta,G)$ is a partition of $V(G)$ and $v_0\in S_1$, we must have $\{w,v_0\}\subseteq S_1$. Note also  $z$ is $\theta$-positive in $G\setminus v_0$ (recall that $H$ is $\theta$-super positive). Using a similar argument, we can show that $\{z,v_0\}\subseteq S_1$. By Lemma \ref{subset_extreme} and Lemma \ref{barrier_is_extreme}, we conclude that $\{w,z\}\subseteq S_1\setminus (\{v_0\}\cup A_{\theta}(G\setminus v_0))$ is a $\theta$-extreme set in $(G\setminus v_0)\setminus A_{\theta}(G\setminus v_0)$. This implies that $\textnormal{mult} (\theta, ((G\setminus v_0)\setminus A_{\theta}(G\setminus v_0))\setminus wz)=3+\vert A_{\theta}(G\setminus v_0)\vert$, contradicting the last sentence of the preceding paragraph. Hence $P_{\theta}(G\setminus v)=\varnothing$ for all $v\in V(G)$ and $G$ is $\theta$-elementary.
\end{proof}

\begin{lm}\label{partition_property} Suppose $G$ is $\theta$-elementary. Then for each $\varnothing\neq X\subseteq S\in \mathfrak P(\theta,G)$, $A_{\theta}(G\setminus X)=S\setminus X$ and $P_{\theta}(G\setminus X)\cup N_{\theta}(G\setminus X)=\varnothing$.
\end{lm}

\begin{proof} Let $x\in X$. Then $P_{\theta}(G\setminus x)=\varnothing$. By Theorem \ref{neutral_empty_positive}, $N_{\theta}(G\setminus x)=\varnothing$. Now by Lemma \ref{barrier_special}, $S\setminus \{x\}=A_{\theta}(G\setminus x)$ so that $X\setminus \{x\} \subseteq S \setminus \{x\} = A_{\theta}(G\setminus x)$. By Theorem \ref{Gallai_Edmond_vertex}, we conclude that $A_{\theta}(G\setminus X)=S\setminus X$ and $P_{\theta}(G\setminus X)\cup N_{\theta}(G\setminus X)=\varnothing$.
\end{proof}

\begin{cor}\label{partition_charac} Suppose $G$ is $\theta$-elementary. Let $S\subseteq V(G)$. Then $S\in \mathfrak P(\theta,G)$ if and only if $G\setminus S$ has exactly $\vert S\vert$ components and each is $\theta$-critical.
\end{cor}

\begin{proof} Suppose $G\setminus S$ has exactly $\vert S\vert$ components and each is $\theta$-critical. Then $c_{\theta}(G\setminus S)=\vert S\vert$ and $S$ is a barrier set. Hence $S\in \mathfrak P(\theta,G)$.

The other implication follows from Lemma \ref{partition_property} and Corollary \ref{Gallai_Edmond_decom}.
\end{proof}

\section{$1$-elementary cycles}

We shall need the following lemmas.

\begin{lm}\label{hamiltonian_path}\textnormal {\cite[Corollary 4.4]{KW2}}
Suppose $G$ has a Hamiltonian path $P$ and $\theta$ is a root of $\mu(G,x)$. Then every vertex of $G$ which is not $\theta$-essential must be $\theta$-special.
\end{lm}

\begin {lm}\label{value_path} Let $p_n$ be a path with $n\geq 1$ vertices. Then
\begin {equation}
\mu (p_n, 1)=
\begin {cases}
1, &\ \textnormal {if $n\equiv 0$ or $1\ \mod 6$;}\\
-1, &\ \textnormal {if $n\equiv 3$ or $4\ \mod 6$;}\\
0, &\ \textnormal {otherwise.}
\end {cases}\notag
\end {equation}
\end {lm}

\begin {proof} Note that for $t\geq 2$, $\mu (p_t, x)=x\mu (p_{t-1}, x)-\mu (p_{t-2}, x)$ (part (c) of Theorem \ref{matching_property}), where we define $\mu (p_0, x)=1$. Therefore $\mu (p_t, 1)=\mu (p_{t-1}, 1)-\mu (p_{t-2}, 1)$. Now $\mu (p_1, 1)=1$. So, $\mu (p_2, 1)=0$, and recursively we have $\mu (p_3, 1)=-1$, $\mu (p_4, 1)=-1$ and $\mu (p_5, 1)=0$. By induction the lemma holds.
\end {proof}

\begin {lm}\label{value_cycle}  Let $C_n$ be a cycle with $n\geq 3$ vertices.  Then
\begin {equation}
\mu (C_n, 1)=
\begin {cases}
1, &\ \textnormal {if $n\equiv 1$ or $5\ \mod 6$;}\\
-1, &\ \textnormal {if $n\equiv 2$ or $4\ \mod 6$;}\\
2, &\ \textnormal {if $n\equiv 0\ \mod 6$;}\\
-2, &\ \textnormal {if $n\equiv 3\ \mod 6$.}
\end {cases}\notag
\end {equation}
\end {lm}

\begin {proof} By part (c) of Theorem \ref{matching_property}, $\mu (C_n,1)=\mu (p_{n-1},1)-2\mu (p_{n-2},1)$. The lemma follows from Lemma \ref{value_path}.
\end {proof}

\begin{thm}\label{1_elementary_cycle} A cycle $C_n$ is 1-elementary if and only if $n=3k$ for some $k\in\mathbb N$.
\end{thm}

\begin{proof} ($\Rightarrow$) Suppose $C_n$ is 1-elementary. Then for any $v\in V(C_n)$, $C_n\setminus v=p_{n-1}$.
By Lemma \ref{value_path}, $\textnormal{mult} (1,p_{n-1})>0$ if and only if $n-1\equiv 2$ or $5 \mod 6$. Thus $n=3k$ for some $k\in\mathbb N$.

\noindent
($\Leftarrow$) Suppose $n=3k$ for some $k\in\mathbb N$. By Lemma \ref{value_cycle}, $\textnormal{mult} (1,C_{n})=0$. Note that $3k\equiv 3$ or $6 \mod 6$. Therefore $3k-1\equiv 2$ or $5 \mod 6$, and by Lemma \ref{value_path} and Lemma \ref{basic_inequality}, $\textnormal{mult} (1,C_{n}\setminus v)=\textnormal{mult} (1,p_{n-1})=1$ for all $v\in V(C_n)$. Thus $C_n$ is 1-super positive. By Lemma \ref{hamiltonian_path}, $P_{1}(C_n\setminus v)=\varnothing$ for all $v\in V(C_n)$. Hence $C_n$ is 1-elementary.
\end{proof}

For our next result, let us denote the vertices of $C_{3k}$ by $1,2,3,\dots, 3k$ (see Figure 4).

\begin{center}
\begin{pspicture}(0,0)(6,3)
\cnodeput(1, 1){1}{}
\cnodeput(2, 1){2}{}
\cnodeput(3, 1){3}{}
\cnodeput(6, 1){4}{}
\ncline{1}{2}
\ncline{2}{3}
\nccurve[angle=90]{1}{4}
\psline[linestyle=dashed]{-}(3.2,1)(5.8,1)
\rput(3.5,0){Figure 4.}
\rput(0,1){$C_{3k}=$}
\rput(1,0.7){$1$}
\rput(2,0.7){$2$}
\rput(3,0.7){$3$}
\rput(6,0.7){$3k$}
\end{pspicture}
\end{center}

\begin{cor}\label{1_elementary_cycle_partition} $C_{3k}$ has exactly 3 1-barrier sets, that is
\begin{equation}
\mathfrak P(1,C_{3k})=\{
\{1,4,7,\dots, 3k-2\}, \{2,5,8,\dots, 3k-1\}, \{3,6,9,\dots, 3k\}\}.\notag
\end{equation}
\end{cor}

\begin{proof} Note that $C_{3k}\setminus \{1,4,7,\dots, 3k-2\}$ is a disjoint union of $k$ number of $K_2$ and $K_2$ is 1-critical. So $\{1,4,7,\dots, 3k-2\}$ is a 1-barrier set. Similarly $\{2,5,8,\dots, 3k-1\}$ and $\{3,6,9,\dots, 3k\}$ are 1-barrier sets. It then follows from Theorem \ref{1_elementary_cycle} and Theorem \ref{elementary_characterization} that these are the only 1-barrier sets.
\end{proof}

\section{Decomposition of $\theta$-super positive graphs}

\begin{dfn}\label{base_graph} A set $X \subseteq V(G)$ with $|X|>1$ is said to be \emph{independent} in $G$ if for all $u,v\in X$, $u$ and $v$ are not adjacent to each other. A graph $G$ is said to be $\theta$-\emph{base} if it is $\theta$-super positive and for all $S\in \mathfrak P(\theta,G)$, $S$ is independent.
\end{dfn}

Note that the cycle $C_{3k}$ is $\theta$-base. In fact a connected $\theta$-base graph is $\theta$-elementary.

\begin{thm}\label{base_elementary} A connected $\theta$-base graph is $\theta$-elementary.
\end{thm}

\begin{proof} Let $G$ be $\theta$-base. Suppose it is not $\theta$-elementary. Then $P_{\theta}(G\setminus v)\neq \varnothing$ for some $v\in V(G)$. By Lemma \ref{existence_essential}, $G\setminus v$ has at least one $\theta$-essential vertex.

If $v$ is not a cut vertex of $G$, then $A_{\theta}(G\setminus v)\neq \varnothing$. By Theorem \ref{neutral_empty_positive} and  Corollary \ref{Gallai_Edmond_decom}, $(G\setminus v)\setminus A_{\theta}(G\setminus v)$ has a $\theta$-super positive component, say $H$. Since $G \setminus v$ is connected, there exists $h\in V(H)$ that is adjacent to some element $w \in A_{\theta}(G\setminus v)$. Note that $\{h,w, v\}$ is a $\theta$-extreme set in $G$. By Lemma \ref{subset_extreme}, $\{h,w\}$ is a $\theta$-extreme set in $G$. By Lemma \ref{barrier_contains_extreme}, $\{h,w\}$ is contained in some $S\in \mathfrak P(\theta,G)$, a contrary to the fact that $S$ is independent.

If $v$ is a cut vertex of $G$, then $G \setminus v$ contains a $\theta$-super positive component (for $N_{\theta}(G\setminus v)=\varnothing$ by Theorem \ref{neutral_empty_positive}). Clearly, some vertex in this component, say $u$, is joined to $v$ and $\{u,v\}$ is a $\theta$-extreme set in $G$. Again, by Lemma \ref{barrier_contains_extreme}, $\{u,v\}$ is contained in some $S \in \mathfrak P(\theta, G)$, a contrary to the fact that $S$ is independent.

Hence $P_{\theta}(G\setminus v)=\varnothing$ for all $v\in V(G)$ and $G$ is $\theta$-elementary.
\end{proof}

Note that the converse of Theorem \ref{base_elementary} is not true. Let $G$ be the graph in Figure 3. Note that $\{u_3,u_4\}\in \mathfrak P(1,G)$ but it is not independent.

\begin{lm}\label{decom_super_1} Let $G$ be $\theta$-super positive and $e=(u,v)\in E(G)$ such that  $\{u,v\}$ is a $\theta$-extreme set in $G$. Let $G'$ be the graph obtained by removing the edge $e$ from $G$. Then $G'$ is $\theta$-super positive.
\end{lm}

\begin{proof} Now $\textnormal{mult}(\theta,G\setminus uv)=2$. By part (b) of Theorem \ref{matching_property}, $\mu (  G,x)=\mu (G',x)-\mu (  G\setminus uv, x)$. This implies that $\mu (G',\theta)=\mu ( G,\theta)\neq 0$.

It is left to show that $\mu (G'\setminus w,\theta)=0$ for all $w\in V(G')$. Clearly if $w=u$ or $v$ then $\mu (G'\setminus w,\theta)=\mu (G\setminus w,\theta)=0$. Suppose $w\neq u,v$. By part (b) of Theorem \ref{matching_property} again, $\mu (G\setminus w,x)=\mu (G'\setminus w,x)-\mu (G\setminus wuv, x)$. By Lemma \ref{basic_inequality}, $\textnormal{mult} (\theta, G\setminus uvw)\geq 1$. Therefore $\mu (G'\setminus w,\theta)=\mu (G\setminus w,\theta)=0$. Hence $G'$ is $\theta$-super positive.
\end{proof}

Note that after removing an edge from $G$ as in Lemma \ref{decom_super_1}, $\mathfrak P(\theta,G')\neq \mathfrak P(\theta,G)$ in  general. In Figure 5, $\mathfrak P(1,G)=\{ \{1,4,7\}, \{5,8\}, \{6,9\}, \{2\}, \{3\}\}$. After removing the edge $(1,4)$ from $G$, the resulting graph $G'=C_{9}$. By Corollary \ref{1_elementary_cycle_partition}, $\mathfrak P(1,G')=\{ \{1,4,7\}, \{2,5,8\}, \{3,6,9\}\}$.

\begin{center}
\begin{pspicture}(0,0)(5,3)
\cnodeput(1, 2){1}{}
\cnodeput(2, 2){2}{}
\cnodeput(3, 2){3}{}
\cnodeput(4, 2){4}{}
\cnodeput(5, 2){5}{}
\cnodeput(5, 1){6}{}
\cnodeput(3.5, 1){7}{}
\cnodeput(2.5, 1){8}{}
\cnodeput(1, 1){9}{}
\ncline{1}{2}
\ncline{2}{3}
\ncline{3}{4}
\ncline{4}{5}
\ncline{5}{6}
\ncline{6}{7}
\ncline{7}{8}
\ncline{8}{9}
\ncline{9}{1}
\nccurve[angle=90]{1}{4}
\rput(2.5,0){Figure 5.}
\rput(0,1.5){$G=$}
\rput(0.7,2){$1$}
\rput(2,2.3){$2$}
\rput(3,2.3){$3$}
\rput(4.3,2.2){$4$}
\rput(5,2.3){$5$}
\rput(5,0.7){$6$}
\rput(3.5,0.7){$7$}
\rput(2.5,0.7){$8$}
\rput(1,0.7){$9$}
\end{pspicture}
\end{center}

We shall need the following lemma.

\begin{lm}\label{path_interlacing}\textnormal {\cite[Corollary 2.5]{G}} For any root $\theta$ of $\mu(G,x)$ and a path $p$ in $G$,
\[ \textnormal{mult}  (\theta, G \setminus p) \ge \m(\theta, G)-1. \]
\end{lm}

\begin{lm}\label{decom_super_unique} Let $G$ be $\theta$-super positive and $e_1=(u,v)\in E(G)$ with $\{u,v\}$ is a $\theta$-extreme set. Let $G'=G-e_1$ and $e_2=(w,z)\in E(G')$. Then $\{w,z\}$ is a $\theta$-extreme set in $G'$ if and only if it is a $\theta$-extreme set in $G$.
\end{lm}

\begin{proof} {\bf Case 1}. Suppose  $e_1$ and $e_2$ have a vertex in common, say $w=u$. Then $G'\setminus wz=G\setminus wz$.

($\Rightarrow$) Suppose $\{w,z\}$ is a $\theta$-extreme set in $G'$.  By Lemma \ref{decom_super_1}, $\textnormal{mult} (\theta, G')=0$. Therefore $\textnormal{mult} (\theta, G\setminus wz)=\textnormal{mult} (\theta, G'\setminus wz)=2$ and $\{w,z\}$ is a $\theta$-extreme set in $G$.

($\Leftarrow$) The converse is proved similarly.

\noindent
{\bf Case 2}. Suppose  $e_1$ and $e_2$ have no vertex in common. By part (b) of Theorem \ref{matching_property},
\begin{equation}
\mu (G\setminus wz,x)=\mu (  G'\setminus wz,x)-\mu (  G\setminus wzuv, x).\notag
\end{equation}

($\Rightarrow$) Suppose $\{w,z\}$ is a $\theta$-extreme set in $G'$. Then $\textnormal{mult} (\theta, G'\setminus wz)=2$. Now $\textnormal{mult} (\theta, G\setminus uv)=2$ and by Lemma \ref{path_interlacing},  $\textnormal{mult} (\theta, G\setminus uvwz)\geq 1$. So we conclude that $\textnormal{mult} (\theta, G\setminus wz)\geq 1$. On the other hand,  $N_{\theta}(G\setminus w)=\varnothing$ (Theorem \ref{neutral_empty_positive}). Therefore either $\textnormal{mult} (\theta, G\setminus wz)=0$ or $2$. Hence the latter holds and $\{w,z\}$ is a $\theta$-extreme set in $G$.

($\Leftarrow$) Suppose $\{w,z\}$ is a $\theta$-extreme set in $G$.  Then $\textnormal{mult} (\theta, G\setminus wz)=2$. As before we have  $\textnormal{mult} (\theta, G\setminus uvwz)\geq 1$. So we conclude that $\textnormal{mult} (\theta, G'\setminus wz)\geq 1$. On the other hand, by Lemma \ref{decom_super_1}, $G'$ is $\theta$-super positive. Therefore $N_{\theta}(G'\setminus w)=\varnothing$ (Theorem \ref{neutral_empty_positive}), and then either $\textnormal{mult} (\theta, G'\setminus wz)=0$ or $2$. Hence the latter holds and $\{w,z\}$ is a $\theta$-extreme set in $G'$.
\end{proof}

\begin{dfn}\label{edge_extreme} Let $G$ be $\theta$-super positive. An edge $e=(u,v)\in E(G)$ is said to be $\theta$-\emph {extreme} in $G$ if $\{u,v\}$ is a $\theta$-extreme set.
\end{dfn}

The process described in Lemma \ref{decom_super_1}, can be iterated. Let $Y_0=\{e_1,e_2,\dots, e_k\}\subseteq E(G)$ be the set of all $\theta$-extreme edges.  Let $G_1=G-e_1$. Then $G_1$ is $\theta$-super positive (Lemma \ref{decom_super_1}). Let $Y_1$ be the set of all $\theta$-extreme edges in $G_1$. Then by Lemma \ref{decom_super_unique}, $Y_1=Y_0\setminus \{e_1\}$. Now let $G_2=G_1-e_2$. By applying Lemma \ref{decom_super_1} and Lemma \ref{decom_super_unique}, we see that $G_2$ is $\theta$-super positive and the set of all $\theta$-extreme edges in $G_2$ is $Y_2=Y_0\setminus \{e_1, e_2\}$. By continuing this process, after $k$ steps, we see that $G_k=G-e_1e_2\dots e_k$ is $\theta$-super positive and the set of all $\theta$-extreme edges in $G_k$ is $Y_k=\varnothing$. We claim that $G_k$ is a disjoint union of $\theta$-base graphs. Suppose the contrary. Let $H$ be a component of $G_k$ that is not $\theta$-base. Since $G_k$ is $\theta$-super positive, by part (a) of Theorem \ref{matching_property}, we deduce that $H$ is $\theta$-super positive. Therefore there is a $S\in\mathfrak P(\theta,H)$ for which $S$ is not independent. Let $e=(u,v)\in E(H)$ with $\{u,v\}\subseteq S$. By Lemma \ref{barrier_is_extreme} and Lemma \ref{subset_extreme}, $\{u,v\}$ is a $\theta$-extreme set in $H$. This means that $e$ is $\theta$-extreme in $H$, and by part (a) of Theorem \ref{matching_property}, $e$ is $\theta$-extreme in $G_k$, a contrary to the fact that $Y_k=\varnothing$. Hence $H$ is $\theta$-base and we have proved the following theorem.

\begin{thm}\label{decom_super_2} Let $G$ be $\theta$-super positive. Then $G$ can be decomposed into a disjoint union of $\theta$-base graphs by deleting its $\theta$-extreme edges. Furthermore, the decomposition is unique, i.e. the $\theta$-base graphs are uniquely determined by $G$.
\end{thm}

The proof of the next lemma is similar to Lemma \ref{decom_super_1}, and is thus omitted.
\begin{lm}\label{construct_super_1} Let $G$ be $\theta$-super positive and  $\{u,v\}$ is a $\theta$-extreme set with $e=(u,v)\notin  E(G)$. Let $G'$ be the graph obtained by adding the edge $e$ to $G$. Then $G'$ is $\theta$-super positive.
\end{lm}

Using the process described in Lemma \ref{construct_super_1}, we can construct $\theta$-super positive graph from $\theta$-base graphs. Together with Theorem \ref{decom_super_2}, we see that every $\theta$-super positive can be constructed from $\theta$-base graphs.

\begin{cor}\label{construction_positive1} A graph is $\theta$-super positive if and only if it can be constructed from $\theta$-base graphs.
\end{cor}

In the next theorem, we shall extend Theorem \ref{construction_positive}.

\begin{thm}\label{construction_positive2} Let $G_1$ and $G_2$ be two $\theta$-super positive graphs and $S_i\in \mathfrak P(\theta, G_i)$ for $i=1,2$. Let $G$ be the graph obtained by adding the edges $e_1,e_2,\dots, e_m$ to the union of $G_1$ and $G_2$, where each $e_j$ contains a point in $S_1$ and $S_2$. Then $G$ is $\theta$-super positive.
\end{thm}

\begin{proof} We shall prove by induction on $m$. If $m=1$,  we are done by Theorem \ref{construction_positive}. Suppose $m\geq 2$. Assume that it is true for $m-1$. Let $G'$ be the graph obtained by adding the edges $e_1,e_2,\dots, e_{m-1}$ to the union of $G_1$ and $G_2$. By induction $G'$ is $\theta$-super positive. Let $e_m=(v_1,v_2)$ where $v_i\in S_i$. Note that the number of $\theta$-critical components in $G'\setminus (S_1\cup S_2)$ is $c_{\theta}(G'\setminus (S_1\cup S_2))=c_{\theta}(G_1\setminus S_1)+c_{\theta}(G_2\setminus S_2) = |S_{1}|+|S_{2}|$. So $S_1\cup S_2$ is a $\theta$-barrier set in $G'$. By Lemma \ref{barrier_is_extreme} and Lemma \ref{subset_extreme}, $\{v_1,v_2\}$ is a $\theta$-extreme set in $G'$. Therefore by Lemma \ref{construct_super_1}, $G$ is $\theta$-super positive.
\end{proof}

In Figure 6, the graph $G$ is obtained from two $1$-base graphs by adding edges $e_1$ and $e_2$.

\begin{center}
\begin{pspicture}(0,0)(10,5)
\cnodeput(1, 3){1}{}
\cnodeput(2, 3){2}{}
\cnodeput(3, 3){3}{}
\cnodeput(4, 3){4}{}
\cnodeput(5, 3){5}{}
\cnodeput(5, 2){6}{}
\cnodeput(3.5, 2){7}{}
\cnodeput(2.5, 2){8}{}
\cnodeput(1, 2){A}{}
\cnodeput(7, 4){10}{}
\cnodeput(7, 2){11}{}
\cnodeput(8, 3){12}{}
\cnodeput(9, 3){B}{}
\cnodeput(10, 2){14}{}
\cnodeput(10, 4){15}{}
\ncline{1}{2}
\ncline{2}{3}
\ncline{3}{4}
\ncline{4}{5}
\ncline{5}{6}
\ncline{6}{7}
\ncline{7}{8}
\ncline{8}{A}
\ncline{A}{1}
\nccurve[angleA=-45, angleB=-110]{A}{B}
\nccurve[angle=-90]{6}{12}
\ncline{10}{B}
\ncline{10}{12}
\ncline{10}{11}
\ncline{11}{12}
\ncline{12}{15}
\ncline{B}{14}
\ncline{B}{15}
\ncline{14}{15}
\rput(6,0.5){Figure 6.}
\rput(0.5,2.5){$G=$}
\rput(6,1.7){$e_1$}
\rput(4,1){$e_2$}
\end{pspicture}
\end{center}


\begin{thebibliography}{99}

\bibitem{KC} W. Chen and C. Y. Ku, An analogue of the Gallai-Edmonds Structure Theorem for nonzero roots of the matching polynomial, Journal of Combinatorial Theory Series B, article in press.


\bibitem{G0} C. D. Godsil, {\em Algebraic Combinatorics}, Chapman
and Hall, New York (1993).

\bibitem{G} C. D. Godsil, {\em Algebraic matching theory},  The
Electronic Journal of Combinatorics {\bf 2} (1995), \# R8.

\bibitem{HL} O.J. Heilmann and E.H. Lieb, {\em Theory of monomer-dimer system},  Commun. Math. Physics, {\bf 25} (1972), 190-232.

\bibitem{KW} C.Y. Ku and K.B. Wong, {\em Extensions of Barrier Sets to Nonzero Roots of the
Matching Polynomials}, preprint available at http://www.math.nus.edu.sg/$\sim$ matkcy/barrier.pdf.

\bibitem{KW2} C.Y. Ku and K.B. Wong, {\em Maximum Multiplicity of Matching Polynomial Roots and Minimum Path Cover in General Graph}, preprint available at http://www.math.nus.edu.sg/$\sim$ matkcy/MaxMin2Final.pdf.

\bibitem{Lo} L. Lov\'{a}sz and M.D. Plummer, {\em Matching Theory}, Elsevier Science Publishers, Budapest (1986).





\end{thebibliography}
\end{document}